\documentclass{article}
\usepackage{amsfonts}
\usepackage{amssymb}
\usepackage{amsmath}
\usepackage{amssymb}
\usepackage{url}
\usepackage{tikz}
\usepackage{dcolumn}
\usepackage{cool}
\usepackage{tabularx}
\usepackage{numprint}
\newcommand{\ra}[1]{\renewcommand{\arraystretch}{1.2}}
\usepackage{booktabs}
\usepackage{amsmath,amsthm}
\newtheorem{thm}{Theorem}

\usepackage[square,sort&compress,numbers]{natbib}
\usepackage{bm}
\renewcommand{\maketitle}%
{\thispagestyle{empty}
\noindent
\par\vspace*{3cm}\par
{\centering{\Large\textbf\MATCHtitle\par}
\par\vspace*{6pt}\par
{\textrm\MATCHauthor}\par\vspace*{12pt}\par}%
}
\renewenvironment{abstract}%
{\begin{center}\begin{minipage}[t]{1.0\textwidth}
\noindent\small\textbf{Abstract$\colon$}}%
{\end{minipage}\end{center}\par}

\usepackage{alltt} 
\newcommand{\MATCHtitle}
{Quadrature based optimal iterative methods}
\newcommand{\MATCHauthor}
{
}
\begin{document}
\maketitle

\nopagebreak[1]
\npstyleenglish
\npfourdigitnosep

\begin{center}
Sanjay K. Khattri, \footnote{Corresponding authors.\\
E-mails: sanjay.khattri@hsh.no, agarwal@fit.edu}\\
Department of Engineering, Stord Haugesund University College, Norway \\
Ravi P. Agarwal,\\
Department of Mathematical Sciences, Florida Institute of Technology, \\
Melbourne, Florida 32901-6975, U.S.A.
\end{center}

\begin{abstract}
We present a simple yet powerful and applicable quadrature based scheme for constructing optimal iterative methods. According to the, still unproved, Kung-Traub conjecture an optimal iterative method based on $n+1$ evaluations could achieve a maximum convergence order of $2^n$. Through quadrature, we develop optimal iterative methods of orders four and eight. The scheme can be further applied to develop iterative methods of even higher order. Computational results demonstrate that the developed methods are efficient as compared with many well known methods. 
\vspace{0.15cm}

{\bf Mathematics Subject Classification (2000):} 65H05 $\cdot$ 65D99 $\cdot$ 41A25 
\vspace{0.15cm}

{\bf Keywords:} iterative methods; fourth order; eighth order; quadrature; Newton; convergence; nonlinear; optimal.
\end{abstract}
\section{Introduction}\label{intro}
Many problems in science and engineering require solving nonlinear equation
\begin{equation}
f(x) = 0, \label{eq1}
\end{equation}
[1-13]. One of the best known and probably the most used method for solving the preceding equation is the Newton's method. The classical Newton method is
given as follows ({\bf NM})
\begin{equation}
x_{n+1} = x_{n} - \dfrac{f(x_n)}{f^\prime(x_n)},\qquad n = 0,1,2,3,\ldots, \quad \text{and}\quad  \vert{f^\prime(x_n)}\vert\neq{0}.\label{eq:newton}
\end{equation}
The Newton's method converges quadratically [1-13]. There exists numerous modifications of the Newton's method which improve the convergence rate (see [1-21] and references therein). This work presents a new quadrature based scheme for constructing optimal iterative methods of various convergence orders. According to the Kung-Traub conjecture an optimal iterative method based upon $n+1$ evaluations could achieve a convergence order of $2^n$. Through the scheme, we construct optimal fourth order and eighth order iterative methods. Fourth order method requests three function evaluations while the eighth order method requests four function evaluations during each iterative step. The next section presents our contribution.
\section{Quadrature based scheme for constructing iterative methods}
Our motivation is to develop a scheme for constructing optimal iterative methods. To construct higher order method from the Newton's method \eqref{eq:newton}, we use the following generalization of the Traub's theorem (see \citep[Theorem 2.4]{traub} and \citep[Theorem 3.1]{pet}).
\begin{thm}\label{thm:1}
Let $g_1(x)$, $g_2(x)$,$\ldots$,$g_s(x)$ be iterative functions with orders $r_1$, $r_2$,$\ldots$,$r_s$, respectively. Then the composite iterative functions $$g(x) = g_1(g_2(\cdots(g_z(x))\cdots))$$
define the iterative method of the orders $r_1r_2r_3\cdots\,r_s$.
\end{thm}
From the preceding theorem and the Newton method \eqref{eq:newton}, we consider the fourth order modified double Newton method
\begin{equation}
\left\{ \begin{array}{llll}
y_n &=& x_n - \dfrac{f(x_n)}{f^\prime(x_n)}, \\
x_{n+1} &=& y_n - \dfrac{f(y_n)}{f^\prime(y_n)}. \label{eq:doublenewton}
\end{array}\right.
\end{equation}
Since the convergence order of the double Newton method is four and it requires four evaluations during each step. Therefore, according to the Kung and Traub conjecture, for the double Newton method to be optimal it must require only three function evaluations. By the Newton's theorem the derivative in the second step of the dobule Newton method can be expressed as
\begin{equation}
f^\prime(y_n) = f^\prime(x_n) + \int_{x_n}^{y_n}\,f^{\prime\prime}(t)\,\mathrm{d}t,\label{eq:newtontheorem}
\end{equation}
let us approximate the integral in the preceding equation as follows
\begin{equation}
\int_{x_n}^{y_n}\,f^{\prime\prime}(t)\,\mathrm{d}t = \omega_1\,f(x_n) + \omega_2 \,f(y_n) + \omega_3\,f^\prime(x_n).\label{eq:1}
\end{equation}
To determine the real constants  $\omega_1$, $\omega_2$ and $\omega_3$ in the preceding equation, we consider the equation is valid for the three functions: $f(t) = \textrm{constant}$, $f(t) = t$ and $t(t) = t^2$. Which yields the equations
\begin{equation}
\left\{ \begin{array}{llll}
\omega_1 + \omega_2 &= 0, \\
\omega_1\,x_n + \omega_1\,y_n + \omega_3  &= 0, \\
\omega_1\,x_n^2+\omega_2\,y_n^2+ \omega_3\,2\,x_n &= 2\,(y_n-x_n).
\end{array}\right.
\end{equation}
Solving the preceding equations and substituting the values in the equations \eqref{eq:newtontheorem} and \eqref{eq:1}, we obtain
\begin{equation}
f^\prime(y_n) = 2\,\left(\dfrac{f(y_n)-f(x_n)}{y_n-x_n}\right)-f^\prime(x_n).
\end{equation}
Combining the double Newton method and preceding approximation for the derivative, we propose the method ({\bf {M-4}})
\begin{equation}
\left\{ \begin{array}{llll}
y_n &=& x_n - \dfrac{f(x_n)}{f^\prime(x_n)}, \\
x_{n+1} &=& y_n - \dfrac{f(y_n)}{2\,\left(\dfrac{f(y_n)-f(x_n)}{y_n-x_n}\right)-f^\prime(x_n)}. \label{eq:method1}
\end{array}\right.
\end{equation}
Since the method ({\bf {M-4}}) is fourth order convergent and it requests only three evaluations. Thus according to the Kung-Traub conjecture it is an optimal method. We prove the fourth order convergent behavior of the iterative method \eqref{eq:method1} through the following theorem.
\begin{thm}
Let $\gamma$ be a simple zero of a sufficiently differentiable function $f\colon{\mathbf{D}}\subset{\mathbf{R}}\mapsto{\mathbf{R}}$ in an open interval $\mathbf{D}$. If $x_0$ is sufficiently close to $\gamma$, the convergence order of the method \eqref{eq:method1} is $4$ and the error equation for the method is given as
$$e_{n+1}  = -{\dfrac { \left( c_{{3}}c_{{1}}-c_{{2}}^{2} \right) c_{{2}}}{
c_{{1}}^{3}}}e_{{n}}^{4}+O \left( e_{{n}}^{5} \right) .$$
Here, $e_{n} = x_n-\gamma$, $c_m = f^{m}(\gamma)/\Factorial{m}$ with $m\ge{1}$.
\end{thm}
\begin{proof}
The Taylor's expansion of $f(x)$ and $f^\prime(x_n)$ around the solution $\gamma$ is given as
\begin{flalign}
f(x_n)&=c_1e_{{n}}+c_{{2}}e_{{n}}^{2}+c_{{3}}e_{{n}}^{3}+c_{{4}}e_{{
n}}^{4}+O \left( e_{{n}}^{5} \right) , \label{eq:fx} \\
f^\prime(x_n)&= c_1+2\,c_{{2}}e_{{n}}+3\,c_{{3}} e_{{n}}^{2}+4\,c_{{4}}e_{{n}}
^{3}+O \left( e_{{n}}^{4} \right). \label{eq:dfx}
\end{flalign}
Here, we have accounted for $f(\gamma)=0$. Dividing the equations \eqref{eq:fx} and \eqref{eq:dfx} we obtain
\begin{multline}
\dfrac{f(x_n)}{f^\prime(x_n)} = e_{{n}}-{\dfrac {c_{{2}}}{c_{{1}}}}{e_{{n}}}^{2}-2\,{\dfrac {c_{{3}}c_{
{1}}-{c_{{2}}}^{2}}{{c_{{1}}}^{2}}}{e_{{n}}}^{3}-{\dfrac {3\,c_{{4}}{c_
{{1}}}^{2}-7\,c_{{3}}c_{{2}}c_{{1}}+4\,{c_{{2}}}^{3}}{{c_{{1}}}^{3}}}{
e_{{n}}}^{4}+O \left( {e_{{n}}}^{5} \right). \label{eq:fxdfx}
\end{multline}
From the first step of the method \eqref{eq:method1} and the equations \eqref{eq:fx} and \eqref{eq:dfx}, we obtain
\begin{equation}
y_n = \gamma+{\dfrac {c_{{2}}}{c_{{1}}}}{e_{{n}}}^{2}+2\,{\dfrac {c_{{3}}c_{{
1}}-{c_{{2}}}^{2}}{{c_{{1}}}^{2}}}{e_{{n}}}^{3}+{\dfrac {3\,c_{{4}}{c_{
{1}}}^{2}-7\,c_{{2}}c_{{3}}c_{{1}}+4\,{c_{{2}}}^{3}}{{c_{{1}}}^{3}}}{e
_{{n}}}^{4}+O \left( {e_{{n}}}^{5} \right). \label{eq:y} 
\end{equation}
By the Taylor's expansion of $f(y_n)$ around $x_n$ and using the fist step of the method \eqref{eq:method1}, we get
\begin{equation}
f(y_n) = f(x_n) + f^{\prime}(x_n)\left(-\dfrac{f(x_n)}{f^\prime(x_n)}\right) + \dfrac{1}{2} f^{\prime\prime}(x_n) \left(-\dfrac{f(x_n)}{f^\prime(x_n)}\right)^2+\cdots,
\end{equation}
the successive derivatives of $f(x_n)$ are obtained by differentiating \eqref{eq:dfx} repeatedly. Substituting these derivatives and using the equations \eqref{eq:fxdfx} into the former equation
\begin{equation}
f(y_n) = c_{{2}}e_{{n}}^{2}+2\,{\frac {c_{{3}}c_{{1}}-{c_{{2}}}^{2}}{c_{{1}}
}}e_{{n}}^{3}+{\frac {3\,c_{{4}}{c_{{1}}}^{2}-7\,c_{{2}}c_{{3}}c_{{1
}}+5\,{c_{{2}}}^{3}}{{c_{{1}}}^{2}}}e_{{n}}^{4}+O \left( e_{{n}}^{
5} \right). \label{eq:fy}
\end{equation}
Finally substituting from the equations \eqref{eq:fx}, \eqref{eq:dfx} and \eqref{eq:fy} into the second step of the contributed method \eqref{eq:method1}, we obtain the error equation for the method
\begin{equation}
e_{n+1}  = -{\dfrac { \left( c_{{3}}c_{{1}}-c_{{2}}^{2} \right) c_{{2}}}{
c_{{1}}^{3}}}e_{{n}}^{4}+O \left( e_{{n}}^{5} \right) . 
\end{equation}
Therefore the contributed method \eqref{eq:method1} is fourth order convergent. This completes our proof.
\end{proof}
To construct optimal eighth order optimal method, we consider the method
\begin{equation}
\left\{ \begin{array}{llll}
y_n &=& x_n - \dfrac{f(x_n)}{f^\prime(x_n)}, \\
z_{n} &=& y_n - \dfrac{f(y_n)}{2\,\left(\dfrac{f(y_n)-f(x_n)}{y_n-x_n}\right)-f^\prime(x_n)},\\
x_{n+1} &=& z_n - \dfrac{f(z_n)}{f^\prime(z_n)}. \label{eq:method3}
\end{array}\right.
\end{equation}
Since the order of the method \eqref{eq:method1} is four and order of the method \eqref{eq:newton} is two. Therefore by the theorem \eqref{thm:1} convergence order of the method \eqref{eq:method3}, which is a combination of the methods \eqref{eq:method1} and \eqref{eq:newton},  is eighth. The method \eqref{eq:method3} require five function evaluations therefore, according to the Kung-Traub conjecture, it is not an optimal method. To develop an optimal method let us again express the first derivative by the Newton's theorem
\begin{equation}
f^\prime(z_n) = f^\prime(x_n) + \int_{x_n}^{z_n}f^{\prime\prime}(t)\,\mathrm{d}t, \label{eq:newton2}
\end{equation}
furthermore let us approximate the integral as follows
\begin{equation}
\int_{x_n}^{z_n}f^{\prime\prime}(t)\,\mathrm{d}t = \nu_1\,f(x_n) + \nu_2\,f(y_n) + \nu_3\,f(z_n) + \nu_4\,f^\prime(x_n),\label{eq:newton3}
\end{equation}
to determine the real constants,  $\nu_1$, $\nu_2$, $\nu_3$ and $\nu_4$ in the preceding equation, we consider the equation is valid for the four functions: $f(t) = \textrm{constant}$, $f(t) = t$, $t(t) = t^2$ and $f(t) = t^3$. And, we obtain the four equations
\begin{equation}
\left\{ \begin{array}{llll}
\nu_1 + \nu_2  +\nu_3 &= 0, \\
\nu_1\,x_n  + \nu_2 \, y_n +\nu_3 \,z_n + \nu_4 &= 0, \\
\nu_1 \,x_n^2+ \nu_2\,y_n^ 2  +\nu_3 \,z_n^2 + 2\,\nu_ 4\,x_n&= 2\,(z_n-x_n), \\
\nu_1\,x_n^3 + \nu_2\, y_n^3 +\nu_3\,z_n^3+3\,\nu_4\,x_n^2 &= 3\,(z_n^2-x_n^2). 
\end{array}\right.\nonumber
\end{equation}
from the preceding equation and the equations \eqref{eq:newton2}, \eqref{eq:newton3}, we get
\begin{multline}
f^\prime(z_n) = -{\dfrac {1}{ \left( -y_{{n}}+x_{{n}} \right) ^{2} \left( -z_{{n}}+y_{{
n}} \right)  \left( -z_{{n}}+x_{{n}} \right) }}  \left[\left( -z_{{n}}+y_{{n}} \right) ^{2} \left( -z_{{n}}+x_{{n}} \right) 
 \left( -y_{{n}}+x_{{n}} \right) f^\prime(x_n) \right. \\
 \left. - \left(x_n -y_{{n}} \right) ^{2} \left( 2\,y_{{n}}-3\,z
_{{n}}+x_{{n}} \right) f \left( z_{{n}} \right) + \left(x_n -z_{{n}}\right) ^{3}f \left( y_{{n}} \right) - \left(y_n -z_{{n}}\right) ^{2} \left( 3\,x_n-2\,y_{{n}}-z_{{n}} \right) f \left( x
_{{n}} \right) \right].\nonumber
\end{multline}
Combining the eighth order method \eqref{eq:method3} and the preceding equation, we propose the following optimal eighth order iterative method ({\bf M-8})
\begin{equation}
\left\{ \begin{array}{llll}
y_n &=& x_n - \dfrac{f(x_n)}{f^\prime(x_n)}, \\
z_{n} &=& y_n - \dfrac{f(y_n)}{2\,\left(\dfrac{f(y_n)-f(x_n)}{y_n-x_n}\right)-f^\prime(x_n)},\\
x_{n+1} &=& z_n - \dfrac{f(z_n)}{\left( -y_{{n}}+x_{{n}} \right) ^{2} \left( -z_{{n}}+y_{{
n}} \right)  \left( -z_{{n}}+x_{{n}} \right)} \left[\left( -z_{{n}}+y_{{n}} \right) ^{2} \left( -z_{{n}}+x_{{n}} \right) \right. \\
&& \left.  \left( -y_{{n}}+x_{{n}} \right) f^\prime(x_n)
  - \left(x_n -y_{{n}} \right) ^{2} \left( 2\,y_{{n}}-3\,z
_{{n}}+x_{{n}} \right) f \left( z_{{n}} \right) \right. \\
&& \left. + \left(x_n -z_{{n}}\right) ^{3}f \left( y_{{n}} \right) - \left(y_n -z_{{n}}\right) ^{2} \left( 3\,x_n-2\,y_{{n}}-z_{{n}} \right) f \left( x
_{{n}} \right) \right] . \label{eq:method2}
\end{array}\right.
\end{equation}
Since the method ({\bf {M-8}}) is eighth order convergent and it requests only four evaluations during each iteration. Thus according to the Kung-Traub conjecture it is an optimal method. We prove the eighth order convergent disposition of the iterative method \eqref{eq:method1} through the following theorem.
\begin{thm}
Let $\gamma$ be a simple zero of a sufficiently differentiable function $f\colon{\mathbf{D}}\subset{\mathbf{R}}\mapsto{\mathbf{R}}$ in an open interval $\mathbf{D}$. If $x_0$ is sufficiently close to $\gamma$, the convergence order of the method \eqref{eq:method2} is $8$. The error equation for the method \eqref{eq:method2} is given as
$$e_{n+1} = -{\frac {{c_{{2}}}^{2} \left( c_{{3}}{c_{{1}}}^{3}c_{{4}}-c_{{4
}}{c_{{1}}}^{2}{c_{{2}}}^{2}-{c_{{3}}}^{2}{c_{{1}}}^{2}c_{{2}}+2\,c_{{
3}}c_{{1}}{c_{{2}}}^{3}-{c_{{2}}}^{5} \right) }{{c_{{1}}}^{7}}}e_n ^{8}+O \left( e_n \right)^{9}.$$
\end{thm}
\begin{proof}
Substituting from the equations \eqref{eq:fx}, \eqref{eq:dfx}, \eqref{eq:fxdfx}, \eqref{eq:fy} into the second step of the contributed method \eqref{eq:method2} yields
\begin{equation}
z_{n} = \gamma-{\dfrac { \left( c_{{3}}c_{{1}}-c_{{2}}^{2} \right) c_{{2}}}{
c_{{1}}^{3}}}e_{{n}}^{4}-2\,{\dfrac {c_{{2}}c_{{4}}{c_{{1}}}^{2}+{c
_{{3}}}^{2}{c_{{1}}}^{2}-4\,c_{{3}}c_{{1}}{c_{{2}}}^{2}+2\,{c_{{2}}}^{
4}}{{c_{{1}}}^{4}}}{e_{{n}}}^{5}+O \left( {e_{{n}}}^{6} \right). \label{eq:z} 
\end{equation}
Here, we have used the first step of the method \eqref{eq:method2}. To find a Taylor expansion $f(z_n)$, we consider the Taylor's series of $f(x)$ around $y_n$
\begin{equation}
f(z_n) = f(y_n) + f^\prime(y_n) \,(z_n-y_n)+ \dfrac{f^{\prime\prime}(y_n)}{2}(z_n-y_n)^2+\cdots,
\end{equation}
substituting from equation \eqref{eq:fy} and using the second step of the contributed method \eqref{eq:method2}, we obtain
\begin{equation}
f(z_n) = -{\dfrac { \left( c_{{3}}c_{{1}}-{c_{{2}}}^{2} \right) c_{{2}}}{{c_{{1
}}}^{2}}}{e_{{n}}}^{4}-2\,{\dfrac {c_{{2}}c_{{4}}{c_{{1}}}^{2}+{c_{{3}}
}^{2}{c_{{1}}}^{2}-4\,c_{{3}}c_{{1}}{c_{{2}}}^{2}+2\,{c_{{2}}}^{4}}{{c
_{{1}}}^{3}}}{e_{{n}}}^{5}+O \left( {e_{{n}}}^{6} \right). \label{eq:fz}
\end{equation}
Here, the higher order derivatives of $f(x)$ at the point $y_n$ are obtained by differentiating the equation \eqref{eq:fy} with respect $e_n$. Finally, to obtain the error equation for the method \eqref{eq:method2}, substituting from the equations \eqref{eq:fx}, \eqref{eq:dfx}, \eqref{eq:fy}, \eqref{eq:y}, \eqref{eq:z} \eqref{eq:fz} into the third step of the contributed method \eqref{eq:method2} yields the error equation 
\begin{equation}
e_{n+1} = -{\frac {{c_{{2}}}^{2} \left( c_{{3}}{c_{{1}}}^{3}c_{{4}}-c_{{4
}}{c_{{1}}}^{2}{c_{{2}}}^{2}-{c_{{3}}}^{2}{c_{{1}}}^{2}c_{{2}}+2\,c_{{
3}}c_{{1}}{c_{{2}}}^{3}-{c_{{2}}}^{5} \right) }{{c_{{1}}}^{7}}}e_n ^{8}+O \left( e_n \right)^{9},
\end{equation}
which shows that the convergence order of the contributed method \eqref{eq:method2} is $8$. This completes our proof.
\end{proof}
\section{Numerical examples}
Let us review some well known methods for numerical comparison. Based upon the well known King's method \citep{king} and the Newton's method \eqref{eq:newton},  recently Li et al. constructed a three step and sixteenth order iterative method ({\bf LMM})
\begin{equation}
\left\{ \begin{array}{llll}
y_n &=& x_n - \dfrac{f(x_n)}{f^\prime(x_n)},\\
z_n &=& y_n - \dfrac{2\,f(x_n)-f(y_n)}{2\,f(x_n)-5\,f(y_n)}\,\dfrac{f(y_n)}{f^\prime(x_n)},\\
x_{n+1} &=&z_n- \dfrac{f(z_n)}{f^\prime(z_n)}-\dfrac{2\,f(z_n)-f\left(z_n- \dfrac{f(z_n)}{f^\prime(z_n)}\right)}{2\,f(z_n)-5\,f\left(z_n- \dfrac{f(z_n)}{f^\prime(z_n)}\right)}\,\dfrac{f\left(z_n- \dfrac{f(z_n)}{f^\prime(z_n)}\right)}{f^\prime(z_n)}, 
\end{array}\right.\label{eq:lmmw}
\end{equation}
\citep{sixteenth1}. Based upon the Jarratt's method \citep{jarratt}, recently Ren et al. \citep{six2,rwb00} formulated a sixth order convergent iterative family consisting of three steps and two parameters ({\bf RWB})
\begin{equation}
\left\{ \begin{array}{llll}
y_n &=& & x_n - \dfrac{2}{3} \dfrac{f(x_n)}{f^\prime(x_n)},\\
z_n &=& & x_n - \dfrac{3\,f^\prime(y_n)+ f^\prime(x_n)}{6\,f^\prime(y_n)-2\,f^\prime(x_n)}\,\dfrac{f(x_n)}{f^\prime(x_n)}, \\
x_{n+1} &=& & z_n - \dfrac{(2a-b)f^\prime(x_n)+bf^\prime(y_n)+cf(x_n)} {(-a-b)f^\prime(x_n)+(3a+b)f^\prime(y_n)+cf(x_n)}
\dfrac{f(z_n)}{f^\prime(x_n)},\label{eq:rwb}
\end{array}\right.
\end{equation}
where $a,b,c\in\mathbb{R}$ and $a{\neq}0$. Wang et al. \citep{six2} also developed a sixth order convergent iterative family, based upon the well known Jarratt's method, for solving non-linear equations. Their methods consist of three steps and two parameters ({\bf WKL})
\begin{equation}
\left\{ \begin{array}{llll}
y_n &=& & x_n - \dfrac{2}{3} \dfrac{f(x_n)}{f^\prime(x_n)},\\
z_n &=& & x_n - \dfrac{3\,f^\prime(y_n)+ f^\prime(x_n)}{6\,f^\prime(y_n)-2\,f^\prime(x_n)}\,\dfrac{f(x_n)}{f^\prime(x_n)}, \\
x_{n+1} &=& & z_n - \dfrac{(5\alpha+3\beta)f^\prime(x_n)-(3\alpha+\beta)f^\prime(y_n)} {2\alpha\,f^\prime(x_n)+2\beta\,f^\prime(y_n)}
\dfrac{f(z_n)}{f^\prime(x_n)},\label{eq:wkl}
\end{array}\right.
\end{equation}
where $\alpha,\beta\in\mathbb{R}$ and $\alpha+\beta{\neq}0$. Earlier, Neta \cite{six4} has developed a sixth order convergent family of methods consisting of three steps and one paramete ({\bf NETA})
\begin{equation}
\left\{ \begin{array}{llll}
y_n &=& & x_n - \dfrac{f(x_n)}{f^\prime(x_n)}, \nonumber \\
z_n &=& & y_n - \dfrac{f(x_n)+af(y_n)}{f(x_n)+(a-2)f(y_n)}\,\dfrac{f(y_n)}{f^\prime(x_n)},\nonumber \\
x_{n+1} &=& & z_n -
\dfrac{f(x_n)-f(y_n)}{f(x_n)-3f(y_n)}\,\dfrac{f(z_n)}{f^\prime(x_n)}.\label{eq:neta}
\end{array}\right.
\end{equation}
We may notice that, in the preceding method, with the choice $a=-1$ the correcting factor in the last two steps is the same. Chun and Ham \citep{six1} also developed a sixth order modification of the Ostrowski's method. Their family of methods consist of three-steps ({\bf CH})
\begin{equation}
\left\{ \begin{array}{llll}
y_n &=& & x_n -\dfrac{f(x_n)}{f^\prime(x_n)},  \\
z_n &=& & y_n - \dfrac{f(x_n)}{f(x_n)-2f(y_n)} \,\dfrac{f(y_n)}{f^\prime(x_n)},  \\
x_{n+1} &=& &z_n - \mathcal{H}(u_n) \dfrac{f(z_n)}{f^\prime(x_n)}, \label{eq:ch}
\end{array}\right.
\end{equation}
where $u_n = f(y_n)/f(x_n)$ and $\mathcal{H}(t)$ represents a real valued function satisfying $\mathcal{H}(0)=1,\,\,\mathcal{H}^\prime(0)=2$. In the case
\begin{equation}
\mathcal{H}(t) = \dfrac{1+(\beta+2)t}{1+\beta{t}}
\end{equation}
the third substep is similar to the method developed by Sharma and Guha \cite{six3}. The classical Chebyshev method is expressed as  ({\bf CM}) 
\begin{alignat}{4}
x_{n+1} = x_n - \frac{f(x_n)}{f^\prime(x_n)}\left[1+\frac{1}{2}\frac{f^{\prime\prime}(x_n)f(x_n)}{f^\prime(x_n)^2}\right],\label{eq:chebyshev}
\end{alignat}
\citep{geom} and the classical  Halley method is expressed as ({\bf HM}) 
\begin{alignat}{4}
x_{n+1} = x_n - \frac{f(x_n)}{f^\prime(x_n)}\left[\frac{2\,f^\prime(x_n)^2}{2\,f^\prime(x_n)^2-{f^{\prime\prime}(x_n)f(x_n)}}\right],\label{eq:halley}
\end{alignat} 
\citep{geom}. The convergence order $\xi$ of an iterative method is defined as
\begin{equation}
\lim_{n\to\Infinity}\dfrac{\vert{e_{n+1}}\vert}{\vert{e_{n}}\vert^\xi} = c \neq{0},\nonumber
\end{equation}
and furthermore this leads to the following approximation of the computational order of convergence (COC)
$$\rho \approx \dfrac{\Log{\vert{({x_{n+1}-\gamma})/({x_n-\gamma})}\vert}}{\Log{\vert{({x_{n}-\gamma})/({x_{n-1}-\gamma})}\vert}}.$$
For convergence it is required: $\vert{x_{n+1} - x_{n}}\vert< \epsilon$ and $\vert{f(x_n)}\vert< \epsilon$. Here, $\epsilon = {10}^{-320}$. We test the methods for the following functions
\begin{alignat}{5}
f_1(x) &= x^3+4\,x^2 -10, \qquad &\gamma& &\approx& 1.365.\nonumber \\
f_2(x) &= x\,\exp(x^2)-\sin^2(x)+3\,\cos(x)+5,\qquad &\gamma& &\approx& -1.207.\nonumber \\
f_3(x) &= \sin^2(x)-x^2+1,\qquad &\gamma& &\approx& \pm1.404.\nonumber \\
f_4(x) &= \ArcTan{x}\qquad           &\gamma& &=& 0.\nonumber \\
f_5(x) &= x^4+\Sin{{\pi}/{x^2}}-5, \qquad &\gamma& &=& \sqrt{2}.\nonumber \\
f_{6}(x) &= \E^{(-x^2+x+2)}-1, \qquad &\gamma& &=&  2.0.\nonumber 
\end{alignat}
Computational results are reported in the Table \ref{table:1} and the Table \ref{table:2}. The Table \ref{table:1} presents {(number of functional evaluations, COC during the second last iterative step)} for various methods. While the Table \ref{table:2} reports $\vert{x_{n+1}-x_n\vert}$ for the method {\bf M-8}. Free parameters are randomly selected as: for the method {\bf RWB} $a = b = c = 1$, in the method by Chun et al. ({\bf CH}) $\beta = 1$, in the method {\bf WKL} $\alpha = \beta = 1$, in the method {\bf NETA} $a = 10$.
{\tabcolsep1.550mm
\renewcommand{\ra}[1]{\renewcommand{\arraystretch}{#1}}
\begin{table*}[!ht]
\centering
\ra{1.5}
\small
\begin{tabularx}{\textwidth}{@{}ln{1}{2}cccccccccc@{}}
\toprule
{\bf $f(x)$} & {\bf $x_0$} & {\bf HM}  & {\bf CM} & {\bf LMM} & {\bf NM}  & {\bf RWB} & {\bf NETA} & {\bf CH} & {\bf WKL} & {\bf M-4} & {\bf M-8} \\
\midrule
$f_1(x)$  & $1.2$   &   $(27,3)$ & $(27,3)$ & $(24,16)$  & $(20,2)$  & $(20,3)$ & $(20,6)$ & $(20,6)$ & $(20,3)$ & ${(18,4)}$ & ${\bf (16,8)}$\\
$f_2(x)$  & $-1.0$   &   $(24,3)$ & $(24,3)$ & $(24,15.5)$  & ${ (22,2)}$  & ${ (20,3)}$ & ${ (20,6)}$ & ${(20,6)}$ & ${(20,3)}$ & ${(18,4)}$ & ${ \bf (16,8)}$\\ 
$f_3(x)$  & $1.5$   &   $(21,3)$ & $(21,3)$ & $(18,15.8)$    & $(20,2)$& $(20,3)$ & ${(20,6)}$ & ${(20,6)}$ & $(20,3)$ & $(18,4)$ & ${\bf (16,8)}$\\  
$f_4(x)$  & $0.5$   &   $(21,3)$     & $(21,3)$   &  $(24,24)$   & ${(18,2)}$& ${ (20,3)}$ & ${\bf (16,7)}$ & ${\bf  (16,7)}$ & ${ (20,3)}$ & ${(18,5)}$ & ${\bf (16,11)}$\\ 
$f_5(x)$&$1.3$ &$(24,3)$& $(24,3)$&  $(24,16)$ & ${(20,2)}$ & ${ (20,3)}$ & ${ (20,6)}$ & ${ (20,6)}$ & ${(20,3)}$ & ${(18,4)}$  & ${\bf (16,8)}$\\ 
$f_6(x)$  & $1.2$ &   $(24,3)$   & $(27,3)$ & $(24,16)$   & $(26,2)$ & $(20,3)$ & $(20,6)$ & $(20,6)$ & $(20,6)$ & ${(21,4)}$ & ${\bf (20,8)}$\\ 
\bottomrule
\end{tabularx}
\caption{(number of functional evaluations, COC) for various iterative methods.}\label{table:1}
\end{table*}}

An optimal iterative method for solving nonlinear equations must require least number of function evaluations. In the Table \ref{table:1}, methods which require least number of function evaluations are marked in bold. We acknowledge, through the Table \ref{table:1}, that the contributed methods in this article are showing better performance to the existing methods in the literature.
\begin{table*}[!ht]
\ra{1.3}\tabcolsep2.30mm
\small
\centering
\begin{tabular}{@{}llllll@{}}
\toprule
{\bf $f_1(x)$} & {\bf $f_2(x)$} & {\bf $f_3(x)$} & {\bf $f_4(x)$} & {\bf $f_5(x)$} & {\bf $f_6(x)$}\\
\hline
$\numprint{1.6e{-1}}$   & $\numprint{2.0e{-1}}$    & $\numprint{9.5e{-2}}$      &   $\numprint{4.2e{-1}}$  &   $\numprint{1.1e{-1}}$ &   $\numprint{7.9e{-1}}$\\
$\numprint{3.3e{-9}}$   & $\numprint{2.4e{-6}}$    & $\numprint{4.0e{-10}}$     & $\numprint{1.7e{-5}}$   &   $\numprint{1.1e{-8}}$ &   $\numprint{8.6e{-4}}$\\
$\numprint{6.4e{-71}}$   & $\numprint{1.9e{-45}}$   & $\numprint{6.3e{-77}}$    & $\numprint{3.1e{-55}}$ &   $\numprint{2.8e{-65}}$ &   $\numprint{2.9e{-25}}$\\ 
$\numprint{1.1e{-564}}$   & $\numprint{2.8e{-358}}$   & $\numprint{2.5e{-611}}$ & $\numprint{1.9e{-601}}$ &   $\numprint{5.8e{-518}}$ &   $\numprint{5.8e{-197}}$\\
***********   & ***********   & *********** & *********** &   *********** &   $\numprint{1.3e{-1570}}$\\
\bottomrule
\end{tabular}
\caption{Generated $\vert{x_{n+1}-x_n}\vert$ with $n\ge{1}$ by the method {\bf M-8}. For initialization see the Table \ref{table:1}.}\label{table:2}
\end{table*}
\newpage

\end{document}